\documentclass[a4paper]{amsart}

\usepackage[T1]{fontenc}
\usepackage{amsmath, amssymb, amsthm, mathrsfs, mathtools, graphicx}

\usepackage{varioref}
\usepackage[backref=page]{hyperref}
\usepackage{enumitem, xspace, ifthen, comment}
\usepackage[all]{xy}
\usepackage[nameinlink, capitalize]{cleveref}
\usepackage[dvipsnames]{xcolor}
\usepackage{tikz}
\usepackage{tikz-cd}

\setlist[enumerate]{
  label=(\thethm.\arabic*),
  before={\setcounter{enumi}{\value{equation}}},
  after={\setcounter{equation}{\value{enumi}}},
  itemsep=1ex
}

\setlist[itemize]{
  leftmargin=*,
  topsep=1ex,
  itemsep=1ex,
  label=$\circ$
}

\newtheorem*{thm-plain}{Theorem}
\newtheorem{thm}{Theorem}[section]
\newtheorem{lem}[thm]{Lemma}
\newtheorem{prp}[thm]{Proposition}
\newtheorem{cor}[thm]{Corollary}

\numberwithin{equation}{thm}

\theoremstyle{definition}
\newtheorem{dfn}[thm]{Definition}

\newtheorem*{dfn-plain}{Definition}

\theoremstyle{remark}

\newtheorem*{ntn-plain}{Notation}
\newtheorem{setup}[thm]{Setup}
\newtheorem{rem}[thm]{Remark}

\newtheorem*{rem-plain}{Remark}

\newcommand{\inv}{^{-1}}
\newcommand{\from}{\colon}

\newcommand{\lto}{\longrightarrow}

\newcommand{\x}{\times}

\newcommand{\bij}{\xrightarrow{\,\smash{\raisebox{-.5ex}{\ensuremath{\scriptstyle\sim}}}\,}}
\newcommand{\isom}{\cong}
\newcommand{\defn}{\coloneqq}
\newcommand{\ndef}{\eqqcolon}
\newcommand{\tensor}{\otimes}
\newcommand{\id}{\mathrm{id}}
\newcommand{\pr}{\mathrm{pr}}

\newcommand{\wt}{\widetilde}
\newcommand{\wb}{\overline}

\renewcommand{\d}{\mathrm d}

\newcommand{\delbar}{\overline\partial}

\newcommand{\dual}{^{\smash{\scalebox{.7}[1.4]{\rotatebox{90}{\textup\guilsinglleft}}}}}

\newcommand{\factor}[2]{\left. \raise 2pt\hbox{$#1$} \right/\hskip -2pt \raise -2pt\hbox{$#2$}}

\newdir{ ir}{{}*!/-5pt/@^{(}} 
\newdir{ il}{{}*!/-5pt/@_{(}} 

\newenvironment{Mat}[1]{\begin{pmatrix}#1\end{pmatrix}}

\DeclareMathOperator{\Sym}{Sym}
\DeclareMathOperator{\Aut}{Aut}

\newcommand{\lref}{\labelcref}

\newcommand{\set}[1]{\left\{ #1 \right\}}

\def\rd#1.{\lfloor{#1}\rfloor}
\def\rp#1.{\lceil{#1}\rceil}
\def\tw#1.{\langle{#1}\rangle}

\newcommand{\la}{\langle}
\newcommand{\ra}{\rangle}

\renewcommand{\O}[1]{\mathscr{O}_{#1}}
\newcommand{\Omegap}[2]{\Omega_{#1}^{#2}}

\newcommand{\Omegar}[2]{\Omega_{#1}^{[#2]}}
\newcommand{\T}[1]{\mathcal{T}_{#1}}

\newcommand{\canmod}{\mathrm{can}}

\newcommand{\reg}[1]{{#1}_{\mathrm{reg}}}

\newcommand{\codim}[2]{\mathrm{codim}_{#1}(#2)}

\newcommand{\ccorb}[2]{\mathrm{\tilde c}_{#1} \! \left( #2 \right)}
\newcommand{\cpcorb}[3]{\mathrm{\tilde c}_{#1} \! \left( #3 \right)^{#2}}

\def\Hnought#1.#2.{\mathit{\Gamma} \!\left( #1, #2 \right)}
\def\HH#1.#2.#3.{\mathrm{H}^{#1} \!\left( #2, #3 \right)}
\def\hh#1.#2.#3.{h^{#1} \!\left( #2, #3 \right)}
\def\RR#1.#2.#3.{R^{#1} #2_* #3}
\def\HHc#1.#2.#3.{\mathrm{H}_{\mathrm{c}}^{#1} \!\left( #2, #3 \right)}
\def\Hh#1.#2.#3.{\mathrm{H}_{#1} \!\left( #2, #3 \right)}
\def\Hom#1.#2.{\mathrm{Hom} \!\left( #1, #2 \right)}
\def\End#1.{\mathrm{End} \!\left( #1 \right)}
\def\sHom#1.#2.{\mathscr{H}\!om \!\left( #1, #2 \right)}
\def\sEnd#1.{\mathscr{E}\!nd \!\left( #1 \right)}
\def\Ext#1.#2.#3.{\mathrm{Ext}^{#1} \!\left( #2, #3 \right)}
\def\sExt#1.#2.#3.{\mathscr{E}\!xt^{#1} \!\left( #2, #3 \right)}
\def\Link#1.#2.{\mathrm{Link} \!\left( #1, #2 \right)}

\renewcommand{\H}[1]{\mathbb H^{#1}}

\newcommand{\GL}[2]{\mathrm{GL}(#1, #2)}
\newcommand{\SL}[2]{\mathrm{SL}(#1, #2)}

\newcommand{\Sp}[2]{\mathrm{Sp}(#1, #2)}
\newcommand{\SO}[1]{\mathrm{SO}(#1)}
\newcommand{\SOstar}[1]{\mathrm{SO^*}(#1)}
\newcommand{\SOnull}[1]{\mathrm{SO_0}(#1)}
\newcommand{\U}[1]{\mathrm{U}(#1)}
\newcommand{\SU}[1]{\mathrm{SU}(#1)}

\newcommand{\qe}{quasi-\'etale\xspace}

\DeclareMathOperator{\supp}{supp}

\DeclareMathOperator{\tr}{tr}

\renewcommand{\theta}{\vartheta}
\renewcommand{\phi}{\varphi}

\newcommand{\N}{\ensuremath{\mathbb N}}
\newcommand{\Z}{\ensuremath{\mathbb Z}}
\newcommand{\Q}{\ensuremath{\mathbb Q}}
\newcommand{\R}{\ensuremath{\mathbb R}}
\newcommand{\C}{\ensuremath{\mathbb C}}
\newcommand{\bH}{\ensuremath{\mathbb H}}

\newcommand{\bP}{\ensuremath{\mathbb P}}

\newcommand{\frg}{\mathfrak g}

 \newcommand{\sE}{\mathscr E} 
  
  \newcommand{\sL}{\mathscr L}

\newcommand{\sV}{\mathscr V} \newcommand{\sW}{\mathscr W}

\newcommand{\cA}{\mathcal A} \newcommand{\cB}{\mathcal B} \newcommand{\cC}{\mathcal C}
\newcommand{\cD}{\mathcal D}  
 \newcommand{\cH}{\mathcal H} 
  \newcommand{\cL}{\mathcal L}

 \newcommand{\cW}{\mathcal W} \newcommand{\cX}{\mathcal X}

\definecolor{forrest}{RGB}{81,133,49}
\definecolor{mydarkblue}{RGB}{10,92,153}

\makeatletter
\@namedef{subjclassname@2020}{\textup{2020} Mathematics Subject Classification}
\makeatother

\title{Uniformization of klt pairs by bounded symmetric domains}

\author{Patrick Graf}
\address{Lehrstuhl f\"ur Mathematik I, Universit\"at Bayreuth, 95440 Bayreuth, Germany}
\email{\href{mailto:patrick.graf@uni-bayreuth.de}{patrick.graf@uni-bayreuth.de}}
\urladdr{\href{https://patrickgraf.gitlab.io/en/}{www.graficland.uni-bayreuth.de}}

\author{Aryaman Patel}
\address{Institut de Math\'ematiques de Toulouse, Universit\'e Paul Sabatier, 31062 Toulouse Cedex~9, France}
\email{\href{mailto:aryaman.patel@math.univ-toulouse.fr}{aryaman.patel@math.univ-toulouse.fr}}

\date{October 16, 2024}
\thanks{The first author was funded by the Deutsche Forschungsgemeinschaft (DFG, German Research Foundation) -- Projektnummer 521356266.}

\hypersetup{
  pdfauthor={Patrick Graf and Aryaman Patel},
  pdftitle={Uniformization of klt pairs by bounded symmetric domains},
  pdfkeywords={Uniformization, klt pairs, bounded symmetric domains, Higgs bundles, principal bundles, variations of Hodge structure},
  pdfstartview={Fit},
  pdfpagelayout={OneColumn},
  pdfpagemode={UseNone},
  linktoc=all,
  breaklinks,
  linkcolor=[RGB]{0 0 96},
  citecolor=[RGB]{96 0 0},
  urlcolor=[RGB]{0 96 0},
  colorlinks}

\begin{document}

\begin{abstract}
Given a complex-projective klt pair $(X, \Delta)$ with standard coefficients and such that $K_X + \Delta$ is ample, we determine necessary and sufficient conditions for the pair $(X, \Delta)$ to be uniformized by a bounded symmetric domain.
As an application, we obtain characterizations of orbifold quotients of the polydisc and of the four classical irreducible bounded symmetric domains in terms of Miyaoka--Yau-type Chern equalities.
\end{abstract}

\maketitle

\begingroup
\hypersetup{linkcolor=black}
\tableofcontents
\endgroup

\section{Introduction}

The aim of this article is to derive conditions for a klt pair to be uniformized by a BSD (= bounded symmetric domain).
More precisely, consider a projective klt pair $\cX = (X, \Delta)$ with standard coefficients (also known as a $\cC$-pair~\cite{KebekusRousseau24}), and assume that $K_X + \Delta$ is ample.
We give necessary and sufficient conditions for $\cX \isom \factor\cD\Gamma$ to hold, where $\cD = \factor{G_0}{K_0}$ is a bounded symmetric domain and $\Gamma \subset \Aut(\cD)$ is a cocompact lattice.

Questions of this type have been studied for a long time.
Earlier work on the uniformization problem includes:
\begin{itemize}
\item Aubin~\cite{Aubin78} and Yau~\cite{Yau78} when $X$ is a projective manifold, $\Delta = 0$ and $\cD = \mathbb B^n$ is the unit ball in $\C^n$,
\item Simpson~\cite[Theorem~2]{Simpson88} in the same setting, but with $\cD$ an arbitrary bounded symmetric domain,
\item Greb, Kebekus, Peternell and Taji~\cite[Theorem~1.5]{GKPT20} when $X$ has klt singularities, $\Delta = 0$ and $\cD$ is the unit ball,
\item the second author~\cite[Theorem~1.1]{Patel23} in the setting of~\cite{GKPT20}, but with $\cD$ being arbitrary, and
\item Claudon, Guenancia and the first author~\cite[Theorem~A]{MYeq} in the case that the boundary divisor $\Delta$ may be nonzero and $\cD$ is the unit ball.
\end{itemize}

Compared to~\cite{Patel23}, we drop the assumption that the action of $\Gamma$ on $\cD$ has no fixed points in codimension one.
As a consequence, we need to deal with the codimension one ramification locus of the orbifold universal cover $\cD \to \cX$, which is encoded by $\Delta$.
The setting throughout this article will be as follows.

\begin{setup} \label{std}
Let $\cX = (X, \Delta)$ be an $n$-dimensional klt pair, where $X$ is a complex-projective variety and $\Delta$ has standard coefficients, i.e.~$\Delta = \sum_{i \in I} \big( 1 - \frac1{m_i} \big) \Delta_i$ with integers $m_i \ge 2$ and the $\Delta_i$ irreducible and pairwise distinct.
We assume that the (log) canonical divisor $K_X + \Delta$ is ample.
Furthermore, we denote by $(X^\circ, \Delta^\circ)$ the orbifold locus of $(X, \Delta)$, cf.~\cref{dfn orbifolds}.
\end{setup}

Our main result is then the following.

\begin{thm}[Uniformization for klt pairs] \label{main}
Let $\cX = (X, \Delta)$ be as in \cref{std}.
The following are equivalent:
\begin{enumerate}
\item\label{main.1} $\cX$ is uniformized by a bounded symmetric domain $\cD = \factor{G_0}{K_0}$.
\item\label{main.2} $(X^\circ, \Delta^\circ)$ admits a uniformizing orbi-system of Hodge bundles $(P, \theta)$ corresponding to $G_0$ such that the Chern class equality
\[ \ccorb2{P \x_K \frg} \cdot [K_X + \Delta]^{n-2} = 0 \]
holds.
\end{enumerate}
\end{thm}

\noindent
We refer the reader to \cref{sec hodge groups} for the unexplained notation in the above statement, in particular concerning $G_0$, $K_0$ and $(P, \theta)$.

\subsection*{Outline of proof}

The proof of (the hard direction of) \cref{main} is divided into two parts.
The first part is to show that a klt pair satisfying condition~\lref{main.2} is in fact a complex orbifold, and $(P, \theta)$ is a uniformizing orbi-variation of Hodge structure on this orbifold.

\begin{thm}[Criterion for quotient singularities] \label{main2}
Let $\cX = (X, \Delta)$ be as in \cref{std}, and suppose that $(X, \Delta)$ admits a uniformizing orbi-system of Hodge bundles $(P, \theta)$ for some Hodge group $G_0$ of Hermitian type such that
\[ \ccorb2{P \x_K \frg} \cdot [K_X + \Delta]^{n-2} = 0. \]
Then:
\begin{enumerate}
\item\label{main2.1} $(X, \Delta)$ has only quotient singularities, i.e.~it is a complex orbifold, and
\item\label{main2.2} $(P, \theta)$ is a uniformizing orbi-VHS (corresponding to $G_0$) on $(X, \Delta)$.
\end{enumerate}
\end{thm}

Once we have used \cref{main2} to reduce to the orbifold case, the second part is to prove the following orbifold version of a result due to Simpson~\cite[Proposition~9.1]{Simpson88}.

\begin{thm}[Uniformization for orbifolds] \label{simpson orbifolds}
Let $\cX = (X, \Delta)$ be a compact complex orbifold and $\cD$ a bounded symmetric domain.
The following are equivalent:
\begin{enumerate}
\item The orbifold universal cover $\wt\cX \to \cX$ is isomorphic to $\cD$.
\item $\cX$ admits a uniformizing orbi-VHS $(P, \theta)$ for some Hodge group $G_0$ of Hermitian type with $\cD \isom \factor{G_0}{K_0}$.
\end{enumerate}
\end{thm}

\subsection*{Applications: uniformization by bounded symmetric domains}

We can use \cref{main} to obtain explicit characterizations of orbifold quotients of the polydisc and of the classical irreducible bounded symmetric domains.

\subsubsection*{The polydisc}

The following corollary is a sufficient criterion for a klt pair to be uniformized by the polydisc $\H n \subset \C^n$.
It generalizes~\cite[Corollary~9.7]{Simpson88}, \cite[Thm.~B]{Beauville00} and~\cite[Theorem~5.1]{Patel23}.

\begin{cor}[Uniformization by the polydisc] \label{Hn ample}
Let $\cX = (X, \Delta)$ be a klt pair as in \cref{std}.
Assume that the orbifold tangent bundle $\T{(X^\circ, \Delta^\circ)}$ of the orbifold locus $(X^\circ, \Delta^\circ)$ splits as a direct sum of orbifold line bundles,
\[ \T{(X^\circ, \Delta^\circ)} \isom \bigoplus_{i=1}^n \sL_i. \]
Then the orbifold universal cover $\wt\cX \to \cX$ of $\cX$ is isomorphic to the polydisc $\H n$.
\end{cor}

See \cref{Hn ample converse} for a suitable converse of \cref{Hn ample}.

\subsubsection*{The classical bounded symmetric domains}

By~\cite[Ch.~VIII, Thm.~7.1]{Helgason78}, bounded symmetric domains are the same as Hermitian symmetric spaces of non-compact type.
The irreducible such spaces are classified in~\cite[Ch.~X, \S6.3, p.~518]{Helgason78}.
They are of types A~III, D~III, BD~I ($q = 2$), C~I, E~III and E~VII.
We refer to the first four types as the ``classical'' ones.
In these cases, we have the following uniformization results.

\begin{cor}[Quotients of the Siegel upper half space, type C~I] \label{quot siegel}
Let $\cX = (X, \Delta)$ be a klt pair as in \cref{std}, of dimension $d = \frac12 n ( n + 1 )$, where $n \in \N$.
Then the orbifold universal cover of $\cX$ is the Siegel upper half space
\[ \cH_n \defn \factor{\Sp{2n}\R}{\U n} \]
if and only if there is a rank $n$ orbifold vector bundle $\sE$ on $(X^\circ, \Delta^\circ)$ such that
\begin{enumerate}
\item\label{s1} $\T{(X^\circ, \Delta^\circ)} \isom \Sym^2 \sE$, and
\item\label{s2} we have the Chern class equality
\[ \Big[ 2 \ccorb2{X, \Delta} - \cpcorb12{X, \Delta} + 2 n \, \ccorb2\sE - ( n - 1 ) \, \cpcorb12\sE \Big] \cdot [ K_X + \Delta ]^{d-2} = 0. \]
\end{enumerate}
\end{cor}

\begin{cor}[Quotients of type D~III] \label{quot D III}
Let $\cX = (X, \Delta)$ be a klt pair as in \cref{std}, of dimension $d = \frac12 n ( n - 1 )$, where $n \in \N$.
Then the orbifold universal cover of $\cX$ is the bounded symmetric domain
\[ \cD_n \defn \factor{\SOstar{2n}}{\U n} \]
if and only if there is a rank $n$ orbifold vector bundle $\sE$ on $(X^\circ, \Delta^\circ)$ such that
\begin{enumerate}
\item $\T{(X^\circ, \Delta^\circ)} \isom \bigwedge^2 \sE$, and
\item we have the Chern class equality
\[ \Big[ 2 \ccorb2{X, \Delta} - \cpcorb12{X, \Delta} + 2 n \, \ccorb2\sE - ( n - 1 ) \, \cpcorb12\sE \Big] \cdot [ K_X + \Delta ]^{d-2} = 0. \]
\end{enumerate}
\end{cor}

\begin{rem}
Recall that the group $\SOstar{2n}$ is defined as
\[ \SOstar{2n} \defn \set{ M \in \SL{2n}\C \;\Big|\; M^T M = I_{2n} \text{ and } \overline M^T J M = J }, \]
where $I_n$ is the $n \x n$ identity matrix and $J = \Mat{ 0 & I_n \\ - I_n & 0 }$.
\end{rem}

\begin{cor}[Quotients of type A~III] \label{quot A III}
Let $\cX = (X, \Delta)$ be a klt pair as in \cref{std}, of dimension $d = pq$, where $p, q \in \N$, $pq \ge 2$ and $p \ne q$.
Then the orbifold universal cover of $\cX$ is the bounded symmetric domain
\[ \cA_{p,q} \defn \factor{\SU{p,q}}{\mathrm S \big( \U p \x \U q \big)} \]
if and only if there are orbifold vector bundles $\sV, \sW$ on $(X^\circ, \Delta^\circ)$ of ranks $p$ and $q$, respectively, such that
\begin{enumerate}
\item\label{AIII.1} $\T{(X^\circ, \Delta^\circ)} \isom \sHom\sV.\sW.$, and
\item\label{AIII.2} we have the Chern class equality
\begin{align*}
\Big[ 2(p + q) \big( \ccorb2\sV + \ccorb2\sW \!\big) & - (p + q - 1) \big( \cpcorb12\sV + \cpcorb12\sW \!\big) \\
& + 2 \ccorb1\sV \ccorb1\sW \Big] \cdot [ K_X + \Delta ]^{d-2} = 0.
\end{align*}
\end{enumerate}
\end{cor}

\begin{rem}
If $p = q \ge 2$, then the ``if'' direction of \cref{quot A III} is still valid.
However, the ``only if'' direction only holds up to an orbi-\'etale double cover $X' \to X$.
This is because the automorphism group $\cA_{p,p}$ has two connected components.
The argument is very similar to the one in the proof of \cref{Hn ample converse}.
\end{rem}

\begin{rem}
In the case $p = 1$, the domain $\cA_{1,q}$ is the unit ball $\mathbb B^q \subset \C^q$.
The expression~\lref{AIII.2} simplifies to the Miyaoka--Yau equality and we recover the characterization of ball quotients from~\cite[Theorem~A]{MYeq}.
\end{rem}

\begin{cor}[Quotients of type BD~I] \label{quot BD I}
Let $\cX = (X, \Delta)$ be a klt pair as in \cref{std}, of dimension $d = n$, where $n \ge 3$.
Then the orbifold universal cover of $\cX$ is the bounded symmetric domain
\[ \cB_n \defn \factor{\SOnull{2,n}}{\SO 2 \x \SO n} \]
if and only if there is an orthogonal rank $n$ orbifold vector bundle $\sW$ and an orbifold line bundle $\sL$ on $(X^\circ, \Delta^\circ)$ such that
\begin{enumerate}
\item $\T{(X^\circ, \Delta^\circ)} \isom \sHom\sW.\sL.$, and
\item we have the Chern class equality
\[ \Big[ \ccorb2{\textstyle\bigwedge^2 \sW} + 2 \ccorb2\sW - n \, \cpcorb12\sL \Big] \cdot [ K_X + \Delta ]^{d-2} = 0. \]
\end{enumerate}
\end{cor}

\begin{rem}
The group $\SOnull{2,n}$ is the connected component of the identity of the group $\SO{2,n}$.
A vector bundle $\sW$ is called \emph{orthogonal} if it admits a reduction of structure group to the orthogonal group $\mathrm O(n, \C)$.
Equivalently, there is a symmetric isomorphism $\sW \bij \sW \dual$.
In particular, any orthogonal bundle is self-dual.
\end{rem}

\begin{rem}
In the case $n = 2$, the domain $\cB_2$ is the bidisc $\H 2$, which has already been covered in \cref{Hn ample}.
\end{rem}

\subsection*{Acknowledgements}

We would like to thank the Mathematisches Forschungs\-institut Ober\-wolfach, where this collaboration was started.
We also thank Henri Guenancia for many helpful discussions.

\section{Preliminaries on orbifolds}

We borrow terminology and notation from the article~\cite{MYeq}.
While we refer the reader to~\cite[Section~2]{MYeq} for definitions of objects and structures over orbifolds, we reproduce the most relevant ones in this section.

\subsection{Orbi-structures and developability}

In what follows, let $(X, \Delta)$ be a klt pair with standard coefficients, i.e.~$\Delta = \sum_{i \in I} \big( 1 - \frac1{m_i} \big) \Delta_i$ with integers $m_i \ge 2$ for each $i \in I$.
Set $X^* \defn \reg X \setminus \supp \Delta$.
As one would expect, objects such as sheaves and K\"ahler metrics on a pair $(X, \Delta)$ are defined on local charts.
A collection of local charts is called an \emph{orbi-structure}.

\begin{dfn}[Adapted morphisms, {\cite[Definition~2.1]{MYeq}}]
Let $(X, \Delta)$ be as above, and let $Y$ be a normal variety together with a finite Galois morphism $f \from Y \to X$.
We say that $f$ is:
\begin{itemize}
\item \emph{adapted} to $(X, \Delta)$ if for each $i \in I$, there is an $a_i \in \Z_{\ge 1}$ and a reduced divisor $\Delta'_i$ on $Y$ such that $f^* \Delta_i = a_i m_i \Delta'_i$,
\item \emph{strictly adapted} if $f$ is adapted and $a_i = 1$ for all $i \in I$,
\item \emph{orbi-\'etale} if $f$ is strictly adapted and \'etale over $X^*$.
\end{itemize}
\end{dfn}

If $X$ is compact, the map $f$ is orbi-\'etale if and only if $K_Y = f^*( K_X + \Delta )$ holds.

\begin{dfn}[Orbi-structures, {\cite[Definition~2.2]{MYeq}}] \label{orbis}
An \emph{orbi-structure} on $(X, \Delta)$ is a covering $\set{ U_\alpha }_{\alpha \in J}$ of $X$ by \'etale-open subsets, together with morphisms $f_\alpha \from X_\alpha \to U_\alpha$ adapted to $(U_\alpha, \Delta_\alpha)$, where $X_\alpha$ is a normal variety, for each $\alpha \in J$.
Moreover, we require the following compatibility condition: for all $\alpha, \beta \in J$, the projection map $g_{\alpha\beta} \from X_{\alpha\beta} \to X_\alpha$ is \qe, where $X_{\alpha\beta}$ denotes the normalization of $X_\alpha \x_X X_\beta$.

An orbi-structure is said to be \emph{strict} (resp.~\emph{orbi-\'etale}) if the maps $f_\alpha$ are all strictly adapted (resp.~orbi-\'etale).
It is said to be \emph{smooth} if each $X_\alpha$ is smooth, in which case the maps $g_{\alpha\beta}$ are in fact \'etale.
\end{dfn}

It follows from \cref{orbis} that two given orbi-structures $\set{ X_\alpha, f_\alpha, U_\alpha }_{\alpha \in I}$ and $\set{ Y_\beta, g_\beta, V_\beta }_{\beta \in J}$ over $(X, \Delta)$ are compatible if and only if the maps $h_{\alpha\beta} \from W_{\alpha\beta}\to X_\alpha$ and $h_{\beta\alpha} \from W_{\alpha\beta}\to Y_\beta$ are quasi-\'etale.
Here $W_{\alpha\beta}$ denotes the normalization of $X_\alpha \x_X Y_\beta$.

\begin{dfn}[Orbifold fundamental group, {\cite[Definition~2.14]{MYeq}}]
The topological fundamental group $\pi_1(X, \Delta)$ of $(X, \Delta)$ is defined as
\begin{align*}
\pi_1(X, \Delta) \defn \factor{\pi_1(X^*)}{\la\!\la \gamma_i^{m_i}, \, i \in I \, \ra\!\ra}
\end{align*}
where $\gamma_i$ denotes a loop around $\Delta_i$ for all $i \in I$, and $m_i \in \N$.
\end{dfn}

By~\cite[Theorem~2.16]{MYeq}, there is a one-to-one correspondence between subgroups of $\pi_1(X, \Delta)$ and covers $Y \to X$ branched at most at $\Delta$, under which the trivial subgroup corresponds to the \emph{(orbifold) universal cover} $(\wt X, \wt \Delta)$ of $(X, \Delta)$.

\begin{dfn}[Developability, {\cite[Definition~2.18]{MYeq}}]
A klt pair $(X, \Delta)$ is said to be \emph{developable} if its universal cover $(\wt X, \wt \Delta)$ is a manifold, i.e.~if $\wt X$ is smooth and $\wt\Delta = \emptyset$.
\end{dfn}

\begin{dfn}[Orbifolds] \label{dfn orbifolds}
An \emph{orbifold} is a klt pair $(X, \Delta)$ with standard coefficients that admits a smooth orbi-\'etale orbi-structure.
More generally, the largest open subset of $(X, \Delta)$ that admits a smooth orbi-\'etale orbi-structure is called the \emph{orbifold locus} of $(X, \Delta)$.
We often denote it by $X^\circ$, or more precisely by $(X^\circ, \Delta^\circ)$, where $\Delta^\circ \defn \Delta\big|_{X^\circ}$.
\end{dfn}

\subsection{Orbifolds as Deligne--Mumford stacks}

It is unfortunate that there is no agreed upon definition of an orbifold, and different authors use the term differently.
The definition we use coincides with~\cite[Definition~4.1.1]{BG08}.

It is particularly useful to view an orbifold $(X, \Delta)$ as a smooth DM (= Deligne--Mumford) stack $\cX$ whose generic stabilizer is trivial.
In other words, $\cX$ has an open dense subset isomorphic to an algebraic or analytic space.
A clear advantage of this is that one can talk about morphisms and sheaves on orbifolds without using local charts, as these notions are well-defined for algebraic and analytic stacks.
In order to substitute the word ``orbifold'' with ``smooth DM stack'', it is important to formalize this change in perspective.
This follows essentially from the following classical result of Satake.

\begin{prp}[{\cite[Proposition~4.2.17]{BG08}}]
Every analytic orbifold $\cX \defn (X, \Delta)$ can be presented as the quotient space of a locally free action (i.e.~with finite stabilizers) of a compact Lie group on an analytic space $M$. \qed
\end{prp}

\noindent
The algebraic analog of this statement is also true, replacing ``analytic'' by ``algebraic'' and ``compact Lie group'' by ``reductive linear algebraic group''~\cite[Thm.~4.4]{Kresch09}.

The takeaway is that every analytic (resp.~algebraic) orbifold $\cX$ can be written as $[M / G]$, where $M$ is an analytic (resp.~algebraic) space and $G$ is a compact Lie (resp.~reductive linear algebraic) group.
In other words, orbifolds in the sense of~\cite[Definition~4.1.1]{BG08}) are global quotient stacks.

We refer the reader to~\cite[Definition~2.17]{PS23} for the definition of analytic stacks.
A quotient stack $[M / G]$ naturally admits the structure of an analytic stack (see~\cite[Example~2.18]{PS23}).
We can thus view an analytic (resp.~algebraic) orbifold as an analytic (resp.~algebraic) stack that admits an open covering by substacks of the form $[U / \Gamma]$, where $U$ is an analytic (resp.~algebraic) space and $\Gamma$ is a finite group acting on $U$ by automorphisms.
More formally, we have:

\begin{dfn}[{\cite[Definition~3.10]{PS23}}]
An \emph{analytic orbifold} is an analytic stack $\cX$ together with a family of substacks $\big( f_i \from [U_i / \Gamma_i] \to \cX \big)_{i \in I}$ such that:
\begin{itemize}
\item For each $i \in I$, $U_i$ is a complex manifold, and $\Gamma_i$ is a finite group acting on $U_i$.
\item For each $i \in I$, $f_i \from [U_i / \Gamma_i] \to \cX$ is an open substack.
\item The morphism $\bigsqcup_{i \in I} [U_i / \Gamma_i ] \to \cX$ is a surjective morphism of stacks.
\end{itemize}
\end{dfn}

For the definitions of open substacks and surjective morphisms of stacks, we refer the reader to~\cite[Definitions~3.3 and~2.16]{PS23}.
 
\begin{dfn}
Let $\cX$ be an analytic (resp.~algebraic) stack.
A \emph{coarse moduli space} for $\cX$ is an analytic (resp.~algebraic) space $X$ together with a morphism $m \from \cX \to X$ of stacks such that for all morphisms of stacks $\phi \from \cX \to Y$, where $Y$ is an analytic (resp.~algebraic) space, there exists a unique morphism $\psi \from X \to Y$ such that $\phi = \psi \circ m$.
\end{dfn}

Coarse moduli spaces do not always exist.
However, if the stack $\cX$ is an orbifold $(X, \Delta)$, then $\cX$ admits a coarse moduli space, and it is the expected one.

\begin{lem}
The coarse moduli space of an orbifold $(X, \Delta)$ is $X$.
\end{lem}

\begin{proof}
Consider a smooth orbi-\'etale orbi-structure $\cC = \set{ X_\alpha, f_\alpha, U_\alpha }_{\alpha \in I}$ on the orbifold $(X, \Delta)$.
We have $U_\alpha \isom \factor{X_\alpha}{G_\alpha}$, where $G_\alpha$ is a finite group acting on $X_\alpha$ by automorphisms, for each $\alpha \in I$.
Each orbit space $U_\alpha$ is a geometric quotient, in particular a categorical quotient. 

By~\cite[Corollary~4.3]{PS23}, we know that $U_\alpha$ is the coarse moduli space of the quotient stack $[X_\alpha / G_\alpha]$.
It follows that $X = \bigcup_{\alpha \in I} U_\alpha$ is the coarse moduli space of~$(X, \Delta)$.
\end{proof}

\begin{rem}
A morphism $f \from [X / G] \to Y$ of analytic (resp.~algebraic) stacks, where $Y$ is an analytic (resp.~algebraic) space, is equivalent to a $G$-invariant morphism $g \from X \to Y$ of analytic (resp.~algebraic) spaces (\cite[Proposition~4.2]{PS23}).
The morphism $g$ factors through the geometric quotient $\factor X G$ (when it exists), i.e.~there exists a unique morphism $h \from \factor X G \to Y$ such that $g = h \circ p$, where $p \from X \to \factor X G$ is the quotient map.

In particular, if $\cX$ is an orbifold, a map $f \from \cX \to Y$ can be expressed as a collection of $G_\alpha$-invariant maps $f_\alpha \from X_\alpha \to Y$ that factor through $U_\alpha \isom \factor{X_\alpha}{G_\alpha}$, where as above $\cC = \set{ X_\alpha, f_\alpha, U_\alpha }_{\alpha \in I}$ is a smooth orbi-\'etale orbi-structure on $\cX$ and $G_\alpha$ are the local Galois groups.
\end{rem}

We conclude this discussion by observing that developability for an orbifold $\cX = (X, \Delta)$ is equivalent to the universal covering stack $\wt\cX$ of $\cX$ being a complex manifold.

\subsection{Principal bundles on orbifolds}

Let $G$ be a (real or complex) Lie group.
In the context of uniformization, it is necessary to have a notion of principal $G$-bundles on klt pairs~$(X, \Delta)$ with standard coefficients.

\begin{dfn}[Orbi-principal bundles]
A \emph{orbi-principal $G$-bundle} $P$ on an orbi-structure $\cC = \set{ U_\alpha, f_\alpha, X_\alpha }_{\alpha \in J}$ on the klt pair $(X, \Delta)$ with standard coefficients is the datum of a collection $\set{ P_\alpha }_{\alpha \in J}$ of principal $G$-bundles on each $X_\alpha$, together with isomorphisms $g_{\alpha\beta}^* P_\alpha \isom g_{\beta\alpha}^* P_\beta$ of principal $G$-bundles satisfying the natural cocycle condition on triple overlaps.
\end{dfn}

\begin{dfn}[Associated orbifold vector bundle]
Let $P$ be as above, and let $V$ be a finite-dimensional representation of $G$.
The \emph{associated orbifold vector bundle} $P \x_G V$ is given by
\[ P_\alpha \x_G V \defn \factor{P_\alpha \x V}{\sim} \]
on each $X_\alpha$, where $(p, v) \sim (pg, g\inv v)$ for all $g \in G$.
The rank of this bundle is equal to $\dim V$.
\end{dfn}

We have the following orbifold analog of a classical fact about flat bundles on simply connected manifolds.

\begin{lem} \label{triv}
Let $\wt\cX$ be a simply connected orbifold, i.e.~$\pi_1(\wt\cX) = \set 1$.
Then any principal $G$-bundle $P$ on $\wt\cX$ which admits a flat connection is trivial, i.e.~we have $P \isom \wt\cX \x G$.
\end{lem}

\begin{proof}
This follows from the correspondence between representations of the orbifold (topological) fundamental group into $G$ and flat principal $G$-bundles.
Just as in the case of manifolds, the flat connection on $P$ gives a monodromy representation $\pi_1(\wt\cX) \to G$, which is trivial because $\pi_1(\wt\cX) = \set 1$.
Inverting this construction gives $P \isom \wt\cX \x G$.
\end{proof}

For the definition of systems of Hodge bundles over a complex manifold, we refer to~\cite[p.~898]{Simpson88}.
One can analogously define orbifold versions of Higgs bundles and systems of Hodge bundles over klt pairs with standard coefficients.

\subsection{Hodge groups of Hermitian type} \label{sec hodge groups}

Recall that a bounded symmetric domain $\cD$ can always be expressed as a quotient
\[ \cD = \factor{G_0}{K_0}, \]
where $G_0$ is a \emph{Hodge group of Hermitian type} and $K_0$ is a maximal compact subgroup of $G_0$, unique up to conjugation.
(We refer to~\cite[Sections~8--9]{Simpson88} for definitions.)
Denote by $G$ and $K$ the complexifications of $G_0$ and $K_0$, respectively.
The Lie algebra $\frg$ of $G$ admits a \emph{Hodge decomposition}
\begin{align} \label{hd1}
\frg = \frg^{-1,1} \oplus \frg^{0,0} \oplus \frg^{1,-1}
\end{align}
such that $[\frg^{p,-p}, \frg^{q,-q}] \subset \frg^{p+q,-p-q}$ for all $p, q \in \Z$.

From now on, fix a Hodge group $G_0$ of Hermitian type, and let $K_0$, $G$ and $K$ be as above.

\begin{dfn}[Principal orbi-systems of Hodge bundles] \label{dfn system of hodge}
Let $(X, \Delta)$ be a klt pair with standard coefficients, and fix a smooth orbi-structure $\cC$ on $(X, \Delta)$.
\begin{enumerate}
\item A \emph{principal orbi-system of Hodge bundles $(P, \theta)$ on $(X, \Delta)$ associated to $G_0$} is an orbi-principal $K$-bundle $P$ on $(X, \Delta)$, together with a morphism
\begin{align*}
\theta \from \T{(X, \Delta)} \lto P \x_K \frg^{-1,1}
\end{align*}
of orbifold vector bundles such that for all local sections $u, v$ of $\T{(X, \Delta)}$, we have $[\theta(u), \theta(v)] = 0$.
(For Hodge groups of Hermitian type, the latter condition is automatically satisfied because $[\frg^{-1,1}, \frg^{-1,1}] \subset \frg^{-2,2} = 0$.)
\item A \emph{uniformizing orbi-system of Hodge bundles} is a principal orbi-system of Hodge bundles such that $\theta$ is an isomorphism.
\end{enumerate}
\end{dfn}

\begin{rem} \label{Hodge bundle}
Let $(P, \theta)$ be a principal orbi-system of Hodge bundles on $(X, \Delta)$.
Set $\sE \defn P \times_K \frg$ and $\sE^{p,-p} \defn P \times_K \frg^{p,-p}$ for $p \in \set{-1, 0, 1}$.
The decomposition~\lref{hd1} of $\frg$ induces a decomposition of $\sE$, given by
\begin{align*}
\sE = \bigoplus_{|p| \le 1} \sE^{p,-p}.
\end{align*}
The isomorphism $\theta$ endows $\sE$ with a Higgs field (which we again denote by $\theta$) whose restriction to each summand $\sE^{p,-p}$ is
\begin{align*}
\theta\big|_{\sE^{p,-p}} \from \sE^{p,-p} \lto \sE^{p-1,-p+1} \tensor \Omegap{(X, \Delta)}1,
\end{align*}
given by sending $a \mapsto \big( v \mapsto [ \theta(v), a ] \big)$.
It follows from the Jacobi identity in $\frg$ that $\theta \wedge \theta = 0$.
Thus $(\sE, \theta)$ is in fact an orbi-system of Hodge (vector) bundles.
Hence the name ``\emph{principal} orbi-system of Hodge bundles''.
\end{rem}

\begin{rem} \label{c1=0 automatic}
In the above setting, we always have the vanishing $\ccorb1{\sE} = 0$.
Indeed, the summand $P \times_K \frg^{0,0}$ of $P \times_K \frg$ is self-dual because $\frg^{0,0}$ is self-dual as a $K$-representation.
Since $P \times_K \frg^{-1,1}$ and $P \times_K \frg^{1,-1}$ are dual to each other, it follows that $\ccorb1{P \times_K \frg} = 0$.
\end{rem}

A metric $h$ on an orbifold vector bundle $\set{ V_\alpha }_{\alpha \in I}$ is a metric $h_\alpha$ on each $V_\alpha$ such that the usual compatibility conditions hold.
Chern classes of orbifold vector bundles have been defined and discussed in detail in~\cite[Section~3]{MYeq}, and we use the definitions therein.

\begin{dfn}[Uniformizing orbi-VHS]
Consider the setting of \cref{dfn system of hodge}.
A \emph{uniformizing orbi-variation of Hodge structure} is a uniformizing orbi-system of Hodge bundles $(P, \theta)$ such that the associated bundle $\sE = P \x_K \frg$ is equipped with a flat metric.
\end{dfn}

\noindent
This is equivalent, in our setting, to $\sE$ being polystable and having vanishing first and second (orbifold) Chern classes (cf.~the proof of \cref{main2}).

\section{Proof of \cref{simpson orbifolds}}

The proof of \cref{simpson orbifolds} consists of three steps.
The first step is, given a uniformizing orbi-VHS on $\cX$, to construct a period map $p \from \wt\cX \to \cD = \factor{G_0}{K_0}$.
For this, we follow exactly the procedure of Simpson in the case of manifolds, cf.~\cite[p.~900]{Simpson88}.
The second step is to show that $\wt\cX$ is actually a manifold and that $p$ is an isomorphism.
The third step is to prove the converse.

\begin{ntn-plain}
Let $\pi \from \wt\cX \to \cX$ be the orbifold universal cover of $\cX$.
Without further mention, we will denote the pullback of objects on $\cX$ to $\wt\cX$ by adding a tilde $\wt{\hspace{1ex}}$ on top of the corresponding symbol.
\end{ntn-plain}

\subsection*{Step I: the period map $p$}

Let $(P, \theta)$ be a uniformizing orbi-VHS on $\cX$.
Then by definition, the principal $K$-bundle $P$ admits a reduction in structure group from $K$ to $K_0$, i.e.~there is a $K_0$-bundle $P_H$ such that $P \isom P_H \x_{K_0} K$.
Also by definition, the connection $D_H \defn d_H + \theta + \overline{\theta_H}$ on the $G_0$-bundle $R_H \defn P_H \x_{K_0} G_0$ is flat.
It pulls back to a flat connection on $\wt{R_H} \defn \pi^* R_H$.

Since $\pi_1(\wt \cX) = 0$, it follows from \cref{triv} that there is a global trivialization
\[ \phi \from \wt{R_H} = \wt{P_H} \times_{K_0} G_0 \bij \wt\cX \times G_0. \]
Note that $\pi_1(\cX)$ acts on $\wt\cX$ by deck transformations.
Since $\wt{R_H}$ is a pullback from $\cX$, it admits a $\pi_1(\cX)$-linearization.
The isomorphism $\phi$, however, is not $\pi_1(\cX)$-equivariant.
The following diagram summarizes the situation.
\[ \begin{tikzcd}
\wt{P_H} \times_{K_0} G_0 \arrow{r}{\gamma^*}[swap]{\sim} \arrow{d}{\wr}[swap]{\phi} & \wt{P_H} \times_{K_0} G_0 \arrow{d}{\wr}[swap]{\phi}
&& r \arrow[r, mapsto]\arrow[d, mapsto] & \gamma\cdot r \arrow[d, mapsto]
\\
\wt\cX \times G_0 \arrow[r, dashed] & \wt\cX \times G_0
&& \phi(r) & \phi(\gamma \cdot r)
\end{tikzcd} \]
where $\gamma\cdot r \defn \gamma^* r$ for all $r \in \wt{P_H} \times_{K_0} G_0$ and $\gamma \in \pi_1(\cX)$.
The above diagram can be completed by an isomorphism $\wt\cX \times G_0 \bij \wt\cX \times G_0$ which is given by $\gamma$ in the first component and by a left translation of $G_0$ in the second component.
Thus we obtain a representation
\[ \sigma \from \pi_1(\cX) \to G_0 \]
such that
\begin{equation} \label{478}
\phi_2( \gamma \cdot g ) = \sigma(\gamma) \phi_2(g)
\end{equation}
for all $\gamma \in \pi_1(\cX)$ and $r \in \wt{R_H}$.
Here $\phi_2 \from \wt{R_H} \to G_0$ is the second component of $\phi$.

For each $x \in \wt\cX$, we have the fibre $\wt{P_H}(x) = P_H \big( \pi(x) \big) \subset \wt{R_H}(x) = R_H \big( \pi(x) \big)$, which is mapped by $\phi_2$ to the right $K_0$-coset $p(x) \defn \phi_2 \big( \wt{P_H}(x) \big) \subset G_0$.
This defines a holomorphic map
\[ p \from \wt\cX \lto \cD = \factor{G_0}{K_0}. \]
By~\lref{478}, we have
\[ p(\gamma \cdot x) = \phi_2 \big( \wt{P_H}(\gamma \cdot x) \big) = \phi_2 \big( \gamma \cdot \wt{P_H}(x) \big) = \sigma(\gamma) \phi_2 \big( \wt{P_H}(x) \big) = \sigma(\gamma) p(x) \]
for all $x \in \wt\cX$ and $\gamma \in \pi_1(\cX)$.
Therefore the map $p$ becomes $\pi_1(\cX)$-equivariant if we let $\pi_1(\cX)$ act on $\cD$ via the representation $\sigma$.

The map $p$ is a map of analytic stacks, where we endow the complex manifold $\cD$ with the trivial orbifold structure.
Since $\cD$ is in particular a complex space, $p$ factors through the coarse moduli space $\wt X$ of $\wt\cX$.
That is, we obtain a map $\wb p \from \wt X \to \cD$.

\subsection*{Step II: $p$ is an isomorphism}

By~\cite[Lemma~12.2.3]{CMP17}, the (holomorphic) tangent bundle of $\cD$ is given by
\[ \T \cD = \big[ \frg^{-1,1} \big] \defn G_0 \times_{K_0} \frg^{-1,1}, \]
where $G_0$ denotes the principal $K_0$-bundle $G_0 \to \factor{G_0}{K_0}$ on $\cD$.
By construction of~$p$, we have that
\[ p^* G_0 \isom \wt{P_H}. \]
Let $P_\cD \defn G_0 \x_{K_0} K$ be the extension of structure group from $K_0$ to $K$.
Then $\T \cD = P_\cD \x_K \frg^{-1,1}$ and $p^* P_\cD \isom \wt P$.
In particular, $p^* \T \cD = \wt P \x_K \frg^{-1,1}$.

Let $\big\{ \wt U_\alpha \big\}_{\alpha \in I}$ be an atlas of smooth charts for $\wt\cX$ lying above a family $\{ U_\alpha \}_{\alpha \in I}$ of open subsets of $\wt X$.
Then for each $\alpha \in I$, the factorization of $p$ via $\wb p$ (cf.~Step~I) yields a factorization of $p\big|_{\wt U_\alpha}$ as
\begin{equation} \label{515}
\wt U_\alpha \xrightarrow{\quad s_\alpha \quad} U_\alpha \xrightarrow{\quad t_\alpha \quad} \cD.
\end{equation}
Recall that by assumption there is an isomorphism $\theta \from \T \cX \bij P \x_K \frg^{-1,1}$.
The differential of $p$ is then given by $\wt\theta \from \T{\wt\cX} \bij \wt P \x_K \frg^{-1,1} = p^* \T \cD$.
In particular, $p$ is a local diffeomorphism, i.e.~$p\big|_{\wt U_\alpha} \from \wt U_\alpha \to \cD$ is \'etale.
By~\lref{515}, it follows that also $s_\alpha \from \wt U_\alpha \to U_\alpha$ is \'etale.
Thus $U_\alpha$ is smooth, and $\wt \cX$ is a manifold.

Let $\omega_\cD$ denote the Bergman metric on $\cD$, which is $\Aut(\cD)$-invariant~\cite[Ch.~VIII, Prop.~3.5]{Helgason78}.
Since $p$ is a $\pi_1(\cX)$-equivariant local diffeomorphism, the pullback $p^* \omega_\cD$ is then a $\pi_1(\cX)$-invariant metric on $\wt\cX$.
It therefore descends to an orbifold metric on $\cX$, which is complete since $\cX$ is compact.
This implies that also $\big( \wt\cX, p^* \omega_\cD \big)$ is complete.
By~\cite[Ch.~IV, Thm.~4.6]{KobayashiNomizu96}, it follows that $p \from \wt\cX \to \cD$ is a covering map.
Since $\cD$ is simply connected~\cite[Ch.~VIII, Thm.~4.6]{Helgason78}, $p$ is in fact an isomorphism.

\subsection*{Step III: the converse}

Suppose we have an isomorphism $p \from \wt\cX \bij \cD$.
Let $G_0 = \Aut(\cD)$ be the group of holomorphic automorphisms of $\cD$.
Then $\cD \isom G_0 / K_0$.
Since $\pi_1(\cX)$ acts holomorphically on $\wt\cX$, there is a representation $\sigma \from \pi_1(\cX) \to G_0$ such that $p$ is $\pi_1(\cX)$-equivariant.

As before, let $G_0 \to \cD$ be the natural $K_0$-bundle.
Set $\wt P \defn p^* ( G_0 \x_{K_0} K )$ and $\wt{R_H} \defn p^* ( G_0 \x_{K_0} G_0 )$.
Then:
\begin{enumerate}
\item $\wt P$ descends to a $K$-bundle $P$ on $\cX$, which admits a reduction of structure group $P_H$ to $K_0$.
\item $\wt{R_H} \isom \wt\cX \x G_0$ is trivial and carries a flat connection given by pulling back the Maurer--Cartan form on $G_0$.
This descends to a flat bundle $R_H$ on $\cX$, which satisfies $R_H \isom P_H \x_{K_0} G_0$.
\end{enumerate}
The differential of $p$ is a $\pi_1(\cX)$-invariant isomorphism
\[ \wt\theta \defn \d p \from \T{\wt\cX} \bij \wt P \x_K \frg^{-1,1}. \]
This descends to an isomorphism $\theta \from \T \cX \bij P \x_K \frg^{-1,1}$ of orbifold vector bundles on $\cX$.
The pair $(P, \theta)$ together with the reduction of structure group $P_H$ is therefore a uniformizing orbi-VHS on $\cX$. \qed

\section{Proof of \cref{main2} and of the main result}

The proof of \cref{main2} again consists of three steps and follows almost exactly the first three steps of the proof of~\cite[Thm.~A]{MYeq}.
The proof of \cref{main} then becomes a one-liner.

\subsection*{Step I: The orbi-system of Hodge sheaves $\sE_X$}

By~\cite[Proposition~2.6]{MYeq}, we know that there is a finite morphism $f \from Y \to X$ which is strictly adapted for $(X, \Delta)$ and whose extra ramification in codimension one (i.e.~away from $\supp \Delta$) is supported over a general element $H$ of a very ample linear system on $X$.
Let $N$ be the order of ramification of $f$ along $H$.
Then we have 
\begin{align*}
K_Y = f^* \big( K_X + \Delta + \big( 1 - \textstyle\frac{1}{N} \big) H \big).
\end{align*}
Set $D \defn \Delta + \big( 1 - \frac{1}{N} \big) H$ and define $X^\circ$ to be the largest open subset of $X$ over which the pair $(X, D)$ admits a smooth orbi-\'etale orbi-structure $\cC^\circ = \set{ U_\alpha, p_\alpha, X_\alpha }_{\alpha \in J}$.
We know that $\codim X{X \setminus X^\circ} \ge 3$ by~\cite[Lemma~2.7]{MYeq}.

Set $\Delta^\circ \defn \Delta\big|_{X^\circ}$ and $D^\circ \defn D\big|_{X^\circ}$ and consider the locally free orbi-sheaf $\T{(X^\circ, \Delta^\circ, p_\alpha)}$ on the pair $(X^\circ, D^\circ)$ with respect to the orbi-structure $\cC^\circ$.
We have by hypothesis that 
\begin{align} \label{iso}
\T{(X^\circ, \Delta^\circ, p_\alpha)} \isom P_\alpha \x_K \frg^{-1,1},
\end{align}
where the collection $\set{ P_\alpha }$ makes up the given orbi-principal $K$-bundle $P$ with respect to $\cC^\circ$.
As in \cref{Hodge bundle}, we can consider the orbi-system of Hodge bundles $(\sE_X, \theta)$, where $\sE_X \defn P \x_K \frg$ with respect to $\cC^\circ$.

Consider the reflexive sheaves $\T{(X, \Delta, f)}$ and $\Omegar{(X, \Delta, f)}1$ on $Y$.
Let $Y' \subset Y$ be the largest open subset over which $(Y, \emptyset)$ admits a smooth orbi-\'etale orbi-structure.
We denote by $Y^\circ$ the big open subset of $Y$ given by 
\begin{align*}
Y^\circ \defn f\inv(X^\circ) \cap Y' \subset Y.
\end{align*}
Note that $\reg Y \subset Y^\circ$.
The map $f$ restricts to a map $f^\circ \defn f\big|_{Y^\circ} \from Y^\circ \to X^\circ$, and $f^\circ$ is orbi-\'etale with respect to $D^\circ$ by construction.

\subsection*{Step II: Chern classes of $\sE_X$}

We denote by $\T{(X^\circ, \Delta^\circ, f)}$ and $\Omegar{(X^\circ, \Delta^\circ, f)}1$ the restrictions of the reflexive sheaves $\T{(X, \Delta, f)}$ and $\Omegar{(X, \Delta, f)}1$ to $Y^\circ$.
Let $\set{ V_\beta, q_\beta, Y_\beta }_{\beta \in K}$ be a smooth orbi-\'etale (i.e.~quasi-\'etale, in this case) orbi-structure for the pair $(Y^\circ, \emptyset)$, which exists by~\cite[Proposition~2.6 and Lemma~2.7]{MYeq}.
Consider the following diagram:
\[ \begin{tikzcd}
W_{\alpha\beta} \arrow{rr}{r_{\alpha\beta}} \arrow{dd}[swap]{g_{\alpha\beta}}
&& Y_\beta \arrow{d}{q_{\beta}} \\
&& Y^\circ \arrow{d}{f^\circ} \\
X_\alpha \arrow{rr}{p_\alpha} && X^\circ
\end{tikzcd} \]
Here $W_{\alpha\beta}$ denotes the normalization of $X_\alpha \x_{X^\circ} Y_\beta$.
Since $p_\alpha$ and $f$ are both orbi-\'etale with respect to $D^\circ$, they ramify to the same order along $D^\circ$, and since $q_\beta$ is quasi-\'etale, the smooth orbi-\'etale orbi-structures $\cC^\circ$ and $\set{ f(V_\beta), f \circ q_\beta, Y_\beta}$ on $(X^\circ, D^\circ)$ are compatible.
In particular, the morphisms $g_{\alpha\beta}$ and $r_{\alpha\beta}$ are \'etale, so that $W_{\alpha\beta}$ is smooth.
Moreover, we have $g_{\alpha\beta}^* \T{(X^\circ, \Delta^\circ, p_\alpha)} \isom r_{\alpha\beta}^* \T{Y_\beta}$, where $\T{Y_\beta} \defn q_\beta^{[*]} \T{(X^\circ, \Delta^\circ, f)}$.
It follows by~\lref{iso} that $r_{\alpha\beta}^* \T{Y_\beta} \isom P_{\alpha\beta} \x_K \frg^{-1,1}$, where $P_{\alpha\beta} \defn g_{\alpha\beta}^* P_{\alpha}$ is a principal $K$-bundle on $W_{\alpha\beta}$.

The collection $\set{ P_{\alpha\beta} }$ defines an orbi-principal $K$-bundle on $Y^\circ$ with respect to the smooth orbi-\'etale (= quasi-\'etale) orbi-structure $\set{ q_\beta \circ r_{\alpha\beta}(W_{\alpha\beta}), q_\beta \circ r_{\alpha\beta}, W_{\alpha\beta} }$ on $(Y^\circ, \emptyset)$.
Since $\reg Y$ is smooth, $q_\beta$ is \'etale over $\reg Y$ and therefore we have $(q_\beta \circ r_{\alpha\beta})^ *\T{(X^\circ, \Delta^\circ, f)} \isom P_{\alpha\beta} \x_K \frg^{-1,1}$ over $\reg Y$.
In particular, the orbifold bundle $\T{(X^\circ, \Delta^\circ, f)}\big|_{\reg Y}$ is an honest vector bundle on $\reg Y$, and we can write
\begin{align}\label{hf}
\T{(X^\circ, \Delta^\circ, f)}\big|_{\reg Y} \isom Q \x_{K} \frg^{-1,1} \ndef \sE_{\reg Y}^{-1,1},
\end{align}
where $Q$ is a principal $K$-bundle on $\reg Y$.
In other words, the orbi-principal $K$-bundle $\set{ P_{\alpha\beta} }$ on $Y^\circ$ restricts to a principal $K$-bundle $Q$ on $\reg Y$.

On $\reg Y$, we can also form the associated vector bundles $\sE_{\reg Y}^{0,0} \defn Q \x_{K} \frg^{0,0}$ and $\sE_{\reg Y}^{1,-1} \defn Q \x_{K} \frg^{1,-1} \isom \Omegar{(X^\circ, \Delta^\circ, f)}1\big|_{\reg Y}$, and the direct sum
\begin{align*}
\sE_{\reg Y} \defn Q \times_{K} \frg = \sE_{\reg Y}^{-1,1} \oplus \sE_{\reg Y}^{0,0} \oplus \sE_{\reg Y}^{1,-1}.
\end{align*}
From the inclusion $\Omegar{(X^\circ, \Delta^\circ, f)}1\big|_{\reg Y} \subset \Omegap{\reg Y}1$ and the isomorphism~\lref{hf}, the bundle $(\sE_{\reg Y}, \theta)$ inherits the structure of a system of Hodge bundles, where $\theta \from \sE_{\reg Y} \to \sE_{\reg Y} \tensor \Omegap{\reg Y}1$ is defined as in \cref{Hodge bundle}.

Let $\sE_Y$ denote the reflexive extension of $\sE_{\reg Y}$ to $Y$.
For brevity, set $A \defn K_X + \Delta$.
Following Step~II of the proof of~\cite[Theorem~A]{MYeq}, we arrive at the following equalities of orbifold Chern classes:
\begin{align}
\ccorb2{\sE_Y} \cdot (f^*A)^{n-2} & = \deg(f) \cdot \ccorb2{\sE_X} \cdot A^{n-2}, \label{ceq1} \\
\cpcorb12{\sE_Y} \cdot (f^*A)^{n-2} & = \deg(f) \cdot \cpcorb12{\sE_X} \cdot A^{n-2}. \label{ceq2}
\end{align}

\subsection*{Step III: $(X, \Delta)$ is an orbifold}

Consider again the orbi-system of Hodge bundles $\sE_X$ on $(X^\circ, \Delta^\circ)$.
By assumption and by \cref{c1=0 automatic}, $\sE_X$ satisfies the Chern class equalities
\begin{align*}
\ccorb1{\sE_X} \cdot A^{n-1} = \ccorb2{\sE_X} \cdot A^{n-2} = 0.
\end{align*}
From the equalities~\lref{ceq1} and~\lref{ceq2}, it follows that $\sE_Y$ also satisfies
\begin{align*}
\ccorb1{\sE_Y} \cdot (f^* A)^{n-1} = \ccorb2{\sE_Y} \cdot (f^* A)^{n-2} = 0.
\end{align*}
The surjective map of Lie algebras $\frg^{-1,1} \tensor \frg^{1,-1} \to \frg^{0,0}$ induces a surjection of vector bundles
\begin{align} \label{surj}
\sE_{\reg Y}^{-1,1} \tensor \sE_{\reg Y}^{1,-1} \lto \sE_{\reg Y}^{0,0}.
\end{align}
Let $\sE_Y^{p,-p}$ denote the reflexive extension of $\sE_{\reg Y}^{p,-p}$ from $\reg Y$ to $Y$, for all $p \in \set{ -1, 0, 1 }$.
We have $\sE_Y^{-1,1} \isom \T{(X, \Delta, f)}$ and $\sE_Y^{1,-1} \isom \Omegar{(X, \Delta, f)}1$.
The reflexive sheaf $\sE_Y$ thus decomposes as 
\begin{align*}
\sE_Y \isom \T{(X, \Delta, f)} \oplus \sE_Y^{0,0} \oplus \Omegar{(X, \Delta, f)}1.
\end{align*}
We know from~\cite[Sec.~4.4, proof of Thm.~C]{GT22} that the sheaf $\Omegar{(X, \Delta, f)}1$ is $f^*A$-semistable, and hence so is its dual $\T{(X, \Delta, f)}$.
Note that the (semi-)stability of a reflexive sheaf on $Y$ is equivalent to the (semi-)stability of its restriction to $\reg Y$, in the sense of~\cite[Definition~2.8]{Patel23}.
Thus the vector bundles $\sE_{\reg Y}^{-1,1}$ and $\sE_{\reg Y}^{1,-1}$ are $f^*A$-semistable on $\reg Y$.
Since both sides of surjection~\lref{surj} have degree zero w.r.t.~$f^* A$, it follows that $\sE_{\reg Y}^{0,0}$ is $f^*A$-semistable too.
The assumption that $K_X + \Delta$ is ample implies that $\T{(X, \Delta, f)}$ has negative degree.
Thus by the same arguments as in the proof of~\cite[Proposition~4.4]{Patel23}, it follows that $(\sE_{\reg Y}, \theta)$ is $f^*A$-polystable as a Higgs sheaf on $\reg Y$.

By the correspondence from~\cite[Theorem~5.1]{GKPT20}, we see that the Higgs bundle $\sE_{\reg Y} = Q \x_K \frg$ is induced by a tame, purely imaginary harmonic bundle.
In particular, it carries a flat metric.
Using the charts $W_{\alpha\beta}$ from Step~II, we can consider $\sE_{\reg Y}$ also as the uniformizing orbi-system of Hodge bundles $P \times_K \frg$.
Therefore $(P, \theta)$ is a uniformizing orbi-VHS on $(X^\circ, \Delta^\circ)$.

Let $g \from Z \to Y$ be a maximally quasi-\'etale cover, whose existence is guaranteed by~\cite[Theorem~1.5]{GKP16}.
Then by~\cite[Prop.~3.17]{GKPT20}, the sheaf $g^{[*]} \sE_Y$ is locally free on $Z$.
Set $W \defn X \setminus H \subset X$ and $h \defn f \circ g \from Z \to X$.
On the open subset $h\inv(W) \subset Z$, we have
\begin{align*}
g^{[*]} \sE_Y \isom g^{[*]} \big( \T Y \oplus \sE_Y^{0,0} \oplus \Omegar Y1 \big) \isom \T Z \oplus g^{[*]} \sE_Y^{0,0} \oplus \Omegar Z1.
\end{align*}
In particular, the restriction of $g^{[*]} \sE_Y$ to $h\inv(W)$ contains $\T Z$ as a direct summand.
It follows that $\T Z\big|_{h\inv(W)}$ is locally free, and hence by the solution to the Lipman--Zariski conjecture for klt spaces~\cite{GKKP11, GK13, Dru13} that $h\inv(W)$ is smooth.

By construction, the map $g\big|_{h\inv(W)} \from h\inv(W) \to W$ branches in codimension one exactly at $\Delta|_W$.
It follows from~\cite[Corollary~2.20]{MYeq} that the same holds for its Galois closure $\wt W \to W$, and $\wt W$ is moreover smooth, because it is an \'etale cover of the smooth space $h\inv(W)$.
This implies that the pair $(W, \Delta|_W)$ has only quotient singularities.

Thus far, the divisor $H$ was only assumed to be general in its (basepoint-free) linear system.
The above argument can thus be repeated by choosing general elements $H_1, \dots, H_{n + 1} \in |H|$.
We conclude that the pair $(X, \Delta)$ has only quotient singularites.
This proves~\lref{main2.1}.

We have already seen above that $(P, \theta)$ is a uniformizing orbi-VHS on $(X^\circ, \Delta^\circ)$.
But~\lref{main2.1} means that $X = X^\circ$, therefore~\lref{main2.2} is also proved. \qed

\subsection*{Proof of \cref{main}}

This is now an immediate consequence of \cref{simpson orbifolds} and \cref{main2}. \qed

\section{Chern classes of direct summands}

In this section, we prove the following auxiliary result, which will be used in the proof of \cref{Hn ample}.
The argument is based on ideas from~\cite[Lemma~3.1]{Beauville00}, adapted to the orbifold setting.

\begin{prp} \label{c12 vanishing}
Let $(X, \Delta)$ be as in \cref{std} (but without assuming that $K_X + \Delta$ is ample).
Assume that the orbifold tangent bundle $\T{(X^\circ, \Delta^\circ)}$ of the orbifold locus $(X^\circ, \Delta^\circ)$ has an orbifold line bundle direct summand $L$, i.e.
\[ \T{(X^\circ, \Delta^\circ)} = L \oplus F. \]
Then for any \Q-Cartier divisors $D_1, \dots, D_{n - 2}$ on $X$, we have
\[ \cpcorb12L \cdot [ D_1 ] \cdots [ D_{n - 2} ] = 0. \]
\end{prp}

\subsection{Extension form of a short exact sequence} \label{splitting ses}

Let
\[ 0 \lto S \lto E \xrightarrow{\;\;\pi\;\;} Q \lto 0 \]
be a short exact sequence of orbifold vector bundles on $(X^\circ, \Delta^\circ)$.
We may choose an orbifold hermitian metric $h$ on $E$.
Consider the orthogonal complement $S^\perp$ of $S$ in $E$ w.r.t.~the metric $h$.
Since $S^\perp \isom Q$ via the restriction of $\pi$, this gives rise to a smooth splitting $\gamma \from Q \to E$ of $\pi$.
In other words, $\gamma \in \cA^0(\sHom Q.E.)$ and $\pi \circ \gamma = \id_Q$.
Since $\sHom Q.E.$ is a holomorphic orbifold vector bundle, we can consider $\delbar \gamma \in \cA^{0,1}(\sHom Q.E.)$.

Denote by $\pi_* \from \sHom Q.E. \to \sEnd Q.$ the map $f \mapsto \pi \circ f$, and denote the induced map $\pi_* \from \cA^{0,1}(\sHom Q.E.) \to \cA^{0,1}(\sEnd Q.)$ by the same symbol.
Then we have
\[ \pi_* ( \delbar \gamma ) = \delbar ( \pi_* \gamma ) = \delbar ( \id_Q ) = 0. \]
By the short exact sequence $0 \to \sHom Q.S. \to \sHom Q.E. \xrightarrow{\pi_*} \sEnd Q. \to 0$, it follows that $\delbar \gamma$ is contained in the subspace $\cA^{0,1}(\sHom Q.S.) \subset \cA^{0,1}(\sHom Q.E.)$.
We call
\[ \delbar \gamma \in \cA^{0,1}(\sHom Q.S.) \]
the \emph{extension form} of the sequence $0 \to S \to E \to Q \to 0$ w.r.t.~the metric $h$.

If we pass to the dual sequence $0 \to Q \dual \to E \dual \to S \dual \to 0$ and consider on $E \dual$ the metric induced by $h$, we get (up to a sign) the same extension form under the identification
\[ \cA^{0,1}(\sHom S \dual.Q \dual.) = \cA^{0,1}(\sHom Q.S.). \]

\subsection{The Atiyah sequence}

Let $E$ be an orbifold vector bundle on $(X^\circ, \Delta^\circ)$.
Consider the first jet bundle $J^1 E$, defined as in~\cite[Sec.~4, p.~193]{Atiyah57} for ordinary vector bundles.
(In that paper, the notation $D(E)$ is used.)
Then $J^1 E$ is again an orbifold vector bundle and there is a short exact sequence
\begin{equation} \label{jet ses}
0 \lto E \tensor \Omegap{(X^\circ, \Delta^\circ)}1 \lto J^1 E \lto E \lto 0,
\end{equation}
called the \emph{jet bundle sequence}.
By choosing an orbifold hermitian metric $\wt h$ on $J^1 E$ and applying the construction from \cref{splitting ses}, we get an extension form
\begin{align*}
\delbar \gamma & \in \cA^{0,1} \Big( \sHom E.E \tensor \Omegap{(X^\circ, \Delta^\circ)}1. \Big) \\
& = \cA^{0,1} \Big( \sEnd E. \tensor \Omegap{(X^\circ, \Delta^\circ)}1 \Big) \\
& = \cA^{1,1}( \sEnd E. ).
\end{align*}
This is the curvature tensor of $E$ w.r.t.~the metric $h$ induced by $\wt h$.
In particular, by taking the trace in the endomorphism part we get an element $\tr(\delbar \gamma) \in \cA^{1,1}(X^\circ, \Delta^\circ)$ which equals the first Chern form $\ccorb1{E, h}$ of $E$, an orbifold differential form of type~$(1, 1)$.

\subsection{Direct summands of the tangent bundle}

Assume from now on additionally that $E$ is a direct summand of $\T{(X^\circ, \Delta^\circ)}$, i.e.
\[ \T{(X^\circ, \Delta^\circ)} = E \oplus F. \]
Then~\lref{jet ses} can be written as
\[ 0 \lto \sEnd E. \oplus \sHom F.E. \lto J^1 E \lto E \lto 0. \]
We claim that this sequence splits over the subsheaf $\sHom F.E.$.
That is, we claim there is an $\O{(X^\circ, \Delta^\circ)}$-linear map $D \from J^1 E \to \sHom F.E.$ such that $D \circ i = \id_{\sHom F.E.)}$, where $i \from \sHom F.E. \to J^1 E$ is the inclusion.
Recalling that $J^1 E = E \oplus \big( E \tensor \Omegap{(X^\circ, \Delta^\circ)}1 \big)$ as sheaves of abelian groups, the map $D$ is defined by
\[ D(s, t \tensor \alpha)(u) \defn \alpha(u) \cdot t + \pr_E \big( [s, u] \big), \]
where $s, t \in E$, $\alpha \in \Omegap{(X^\circ, \Delta^\circ)}1$ and $u \in F$ are local sections, $\pr_E \from \T{(X^\circ, \Delta^\circ)} \to E$ is the projection and $[-, -]$ denotes the Lie bracket of (orbifold) vector fields.
It is easy to check that $D$ is indeed $\O{(X^\circ, \Delta^\circ)}$-linear and a splitting over $\sHom F.E.$.
In particular, we have a decomposition
\[ J^1 E = \sHom F.E. \oplus \ker(D) \]
as $\O{(X^\circ, \Delta^\circ)}$-modules.

Now choose the metric $\wt h$ on $J^1 E$ in such a way that this decomposition becomes orthogonal.
Then the extension form of~\lref{jet ses} is contained in the subspace
\[ \cA^{0,1} \Big( \sEnd E. \tensor E \dual \Big) \subset \cA^{0,1} \Big( \sEnd E. \tensor \Omegap{(X^\circ, \Delta^\circ)}1 \Big). \]
This can be seen by dualizing~\lref{jet ses}, applying the construction from \cref{splitting ses} and noting that on the direct summand $\sHom F.E. \dual$, the smooth splitting $\gamma$ is given by the holomorphic map $D \dual$.
Taking the trace, it follows that
\[ \ccorb1{E, h} \in \cA^{0,1}( E \dual ) \subset \cA^{0,1} \big( \Omegap{(X^\circ, \Delta^\circ)}1 \big) = \cA^{1,1}(X^\circ, \Delta^\circ). \]

\subsection{Proof of \cref{c12 vanishing}}

We apply the preceding considerations to $E \defn L$ and obtain that for some metric $h$ on $L$, we have
\[ \ccorb1{L, h} \in \cA^{0,1}( L \dual ). \]
Therefore
\[ \cpcorb12{L, h} = \ccorb1{L, h} \wedge \ccorb1{L, h} \in \cA^{0,2} \big( \! \textstyle\bigwedge^2 L \dual \big) = 0. \]
By multilinearity of the intersection product, we may assume that the divisors $D_i$ are very ample and general in their respective linear system.
Consider the complete intersection surface $S \defn D_1 \cap \cdots \cap D_{n - 2} \subset X^\circ$.
By~\cite[Lemma~3.5]{MYeq}, we obtain
\[ \cpcorb12L \cdot [ D_1 ] \cdots [ D_{n - 2} ] = \int_S \cpcorb12{L, h}\!\big|_S = \int_S 0 = 0, \]
as desired.
(Note that in~\cite[Lemma~3.5]{MYeq}, only the analogous statement for the second Chern class $\mathrm c_2$ is given and it is assumed that $D_1 = \dots = D_{n - 2}$.
However, the proof for $\mathrm c_1^2$ and for a tuple of different divisors works the same way.) \qed

\section{Quotients of the polydisc}

In this section, we give a proof of \cref{Hn ample}.
We also prove a partial converse and discuss the case of minimal models (as opposed to canonical models).

\subsection*{Proof of \cref{Hn ample}}

Consider the Lie groups
\[ G_0 = \SL2\R ^n, \quad K_0 = \U1 ^n, \quad G = \SL2\C ^n, \quad K = (\C^*)^n. \]
The splitting assumption implies that the bundle $\T{(X^\circ, \Delta^\circ)}$ admits a reduction of structure group from $\GL n\C$ to $K$.
Each fiber is isomorphic, as a $K$-representation, to $\frg^{-1,1} \isom \C^n$, where $\frg = \mathfrak{sl}_2(\C)^n$ is the Lie algebra of $G$.
Hence there is an isomorphism
\[ \theta \from \T{(X^\circ, \Delta^\circ)} \bij P \x_K \frg^{-1,1}, \]
where $P$ is the principal $K$-orbibundle on $(X^\circ , \Delta^\circ)$ corresponding to the above reduction of structure group.
In other words, $(P, \theta)$ is a uniformizing orbi-system of Hodge bundles on $(X^\circ, \Delta^\circ)$, corresponding to $G_0$.

We will now check the vanishing condition
\[ \ccorb2{P \x_K \frg} \cdot [K_X + \Delta]^{n-2} = 0 \]
from~\lref{main.2}.
To this end, note that $\frg^{-1,1}$ and $\frg^{1,-1}$ are dual as $K$-representations, while $\frg^{0,0}$ is the trivial representation.
Therefore
\begin{align*}
P \x_K \frg & = \bigoplus_{|i|\le1} P \x_K \frg^{-i,i} \\
& \isom \T{(X^\circ, \Delta^\circ)} \oplus \O{(X^\circ, \Delta^\circ)}^{\oplus n} \oplus \Omegap{(X^\circ, \Delta^\circ)}1 \\
& \isom \bigoplus_{i=1}^n \big( \sL_i \oplus \O{(X^\circ, \Delta^\circ)} \oplus \sL_i \dual \big).
\end{align*}
By the calculus of (orbifold) Chern classes, we get $\ccorb1{P \times_K \frg} = 0$ and
\[ \ccorb2{P \times_K \frg} = - \sum_{i=1}^n \cpcorb12{\sL_i}. \]
From \cref{c12 vanishing}, it follows that
\[ \ccorb2{P \times_K \frg} \cdot [K_X + \Delta]^{n-2} = - \sum_{i=1}^n \cpcorb12{\sL_i} \cdot [K_X + \Delta]^{n-2} = 0 \]
holds.
We may now conclude from \cref{main} that $(X, \Delta)$ is uniformized by the polydisc $\H n = \factor{G_0}{K_0}$. \qed

\subsection*{On the converse of \cref{Hn ample}}

It is not true in general that if the orbifold universal cover is the polydisc, then the orbifold tangent bundle splits into a direct sum of line bundles.
This is due to the fact that $\Aut(\H n)$ is disconnected, and strictly larger than $\Aut(\bH)^n$ acting diagonally.
More precisely, by~\cite[Cor.~on p.~167]{Rudin69} there is a split short exact sequence
\begin{equation} \label{Aut Hn ses}
0 \lto \Aut(\bH)^n \lto \Aut(\H n) \lto S_n \lto 0,
\end{equation}
where $S_n$ denotes the symmetric group on $n$ letters.
Therefore we have a semidirect product decomposition $\Aut(\H n) \isom \Aut(\bH)^n \rtimes S_n$.

A suitable converse of \cref{Hn ample} can now be formulated as follows.

\begin{prp}[Orbifolds with universal cover the polydisc] \label{Hn ample converse}
Let $(X, \Delta)$ be as in \cref{std}.
Assume that the orbifold universal cover $\pi \from \wt X_\Delta \to X$ of $(X, \Delta)$ is the polydisc.
Then there is a finite orbi-\'etale (w.r.t.~$\Delta$) Galois cover $f \from Y \to X$, where $Y$ is a projective manifold with $K_Y$ ample and whose tangent bundle splits as a direct sum of line bundles,
\[ \T Y \isom \bigoplus_{i=1}^n \sL_i. \]
\end{prp}

\begin{proof}
It follows from the assumptions that $(X, \Delta)$ is in fact an orbifold.
The fundamental group $\Gamma \defn \pi_1(X, \Delta)$ is a finitely generated subgroup of $\Aut(\H n)$ and we have $(X, \Delta) \isom \factor{\H n}\Gamma$.
By sequence~\lref{Aut Hn ses}, there is a finite index subgroup $\Gamma_0 \subset \Gamma$ which is contained in $\Aut(\bH)^n$.
Note that $\Aut(\bH) \isom \bP \SL2\R$.
Therefore $\Gamma_0$ is a finitely generated linear group.
By Selberg's lemma~\cite{Alperin87}, there is a torsion-free finite index subgroup $\Gamma_1 \subset \Gamma_0$.
After replacing $\Gamma_1$ with a finite index subgroup again, we may assume that $\Gamma_1 \subset \Gamma$ is a normal subgroup.
Set
\[ Y \defn \factor{\H n}{\Gamma_1} \]
and note that this is a projective manifold.
The universal covering map $\pi \from \H n \to X$ factors as
\[ \H n \xrightarrow{\quad\;\;\quad} Y \xrightarrow{\quad f \quad} X. \]
We need to check the required properties.
\begin{itemize}
\item The map $f$ is finite and Galois with Galois group $\factor\Gamma{\Gamma_1}$.
Also, $f$ is orbi-\'etale with respect to $\Delta$ because $\pi$ has this property and the first map $\H n \to Y$ is \'etale.
\item By the above, we have $K_Y = f^* ( K_X + \Delta )$.
Since $K_X + \Delta$ is ample, so is~$K_Y$.
(Alternatively, one could use the fact that the universal cover of $Y$ is the polydisc.)
\item The tangent bundle of $Y$ splits as a direct sum of line bundles because $\Gamma_1$ acts diagonally on $\H n$.
(See~\cite[proof of Prop.~5.8]{Patel23} for more details.) \qedhere
\end{itemize}
\end{proof}

\subsection*{Minimal models}

We can strengthen \cref{Hn ample} slightly by assuming $K_X + \Delta$ to be only big and nef instead of ample, and showing that the canonical model $(X, \Delta)_\canmod$ is uniformized by the polydisc.

\begin{cor}[Uniformization of minimal models] \label{Hn big nef}
Let $(X, \Delta)$ be as in \cref{std}, but assume only that $K_X + \Delta$ is big and nef.
Furthermore, assume that $\T{(X^\circ, \Delta^\circ)}$ splits as in \cref{Hn ample}.
Then the canonical model $(X, \Delta)_\canmod \ndef (X_\canmod, \Delta_\canmod)$ of the pair $(X, \Delta)$ has the polydisc $\H n$ as its orbifold universal cover.
\end{cor}

\begin{proof}
We only need to show that the canonical model $(X, \Delta)_\canmod$ satisfies the assumptions of \cref{Hn ample}.
Note that $(X_\canmod, \Delta_\canmod)$ is again klt, and $K_{X_\canmod} + \Delta_\canmod$ is ample.

Let $(X_\canmod^\circ, \Delta_\canmod^\circ)$ denote the orbifold locus of $(X, \Delta)_\canmod$, and let $(U, \Delta_\canmod|_U) \subset (X, \Delta)_\canmod$ be the big open subset over which the birational morphism
\[ \pi \from (X, \Delta) \to (X, \Delta)_\canmod \]
is an isomorphism.
Pick a collection of smooth charts $\{ V_\alpha \}_\alpha$ for $(X^\circ_\canmod, \Delta^\circ_\canmod)$.
Then $\{ V'_\alpha \}_\alpha$ is a collection of smooth charts for $\cW \defn (U, \Delta_\canmod|_U) \cap (X^\circ_\canmod, \Delta^\circ_\canmod)$, where $V'_\alpha \subset V_\alpha$ is a big open subset, for each $\alpha$.

We have $\T{V'_\alpha} \isom \bigoplus_{i=1}^n \cL_{\alpha,i}$, where each $\cL_{\alpha,i}$ is a line bundle, because $\pi$ is an isomorphism over $\cW$.
It follows that $\T{V_\alpha} \isom \bigoplus_{i=1}^n \cL'_{\alpha,i}$, where each $\cL'_{\alpha,i}$ is a reflexive sheaf of rank one on $V_\alpha$.
Since each $V_\alpha$ is smooth, $\T{V_\alpha}$ is locally free and each $\cL'_{\alpha,i}$ is in fact a line bundle.
Thus $\T{(X^\circ_\canmod, \Delta^\circ_\canmod)}$ splits as a direct sum of orbifold line bundles as in \cref{Hn ample}.
\end{proof}

\section{Quotients of the classical irreducible BSDs}

In this section, we prove \cref{quot siegel}.
We refer the reader to~\cite[Section~6]{Patel23} for an analogous proof in the setting of klt varieties, which also contains references for the Lie- and representation-theoretic facts used here.
For the other classical irreducible BSDs (Corollaries~\lref{quot D III},~\lref{quot A III} and~\lref{quot BD I}), the proofs are essentially the same as that of \cref{quot siegel} and are thus omitted in order to keep the length of this paper reasonable.
Again, we refer the reader to~\cite[Sections~7--9]{Patel23} for analogous arguments in the non-orbifold setting.

\subsection*{Proof of \cref{quot siegel}}

Suppose $(X, \Delta)$ satisfies conditions~\lref{s1} and~\lref{s2}.
Let $G_0 = \Sp{2n}\R$, $K_0 = \U n$, and $G = \Sp{2n}\C$, $K = \GL n\C$ their complexifications.

Let $P$ be the frame bundle of $\sE$.
Then $P$ is a $K$-principal orbi-bundle on $(X^\circ, \Delta^\circ)$.
Since $G_0$ is a Hodge group of Hermitian type, the Lie algebra $\frg$ of $G$ admits a Hodge decomposition as in~\lref{hd1}.
We have 
\begin{align*}
\frg^{-1,1} \isom \Sym^2 \C^n \isom \C^{n(n+1)/2}
\end{align*}
as $K$-representations.
It follows from~\lref{s1} that there is an isomorphism
\begin{align*}
\theta \from \T{(X^\circ, \Delta^\circ)} \bij P \x_K \frg^{-1,1}
\end{align*}
of orbifold vector bundles.
The pair $(P, \theta)$ is thus a uniformizing orbi-system of Hodge bundles on $(X^\circ, \Delta^\circ)$ for the Hodge group $\Sp{2n}\R$.
It also follows that $P \x_K \frg \isom \Sym^2 ( \sE \oplus \sE \dual )$.
Therefore the Chern class equality~\lref{s2} can be rephrased as
\begin{align*}
\ccorb2{P \x_K \frg} \cdot [ K_X + \Delta ]^{d-2} = 0.
\end{align*}
We conclude from \cref{main} that $(X, \Delta)$ is uniformized by $\cH_n = \Sp{2n}\R / \U n$.

Conversely, suppose the orbifold universal cover of $(X, \Delta)$ is $\cH_n$.
Then we know by \cref{simpson orbifolds} that $(X, \Delta)$ admits a uniformizing orbi-VHS $(P, \theta)$ for the Hodge group $G_0 = \Sp{2n}\R$.
This means that $\T{(X, \Delta)} \isom P \x_K \frg^{-1,1}$, and $\ccorb2{P \x_K \frg} \cdot [ K_X + \Delta ]^{d-2} = 0$, where $K = \GL n\C$ and $\frg = \mathfrak{sp}(2n, \C)$.
Since $\frg^{-1,1} \isom \Sym^2 \C^n$ and $\frg \isom \Sym^2 \big( \C^n \oplus (\C^n) \dual \big)$, it follows that
\begin{align*}
\T{(X, \Delta)} \isom \Sym^2 \sE
\end{align*}
for some vector bundle $\sE$ of rank $n$ on $(X, \Delta)$, and the Chern class equality
\begin{align*}
& \ccorb2{P \x_K \frg} \cdot [ K_X + \Delta ]^{d-2} \; = \\
& \Big[ 2 \ccorb2{X, \Delta} - \cpcorb12{X, \Delta} + 2 n \, \ccorb2\sE - ( n - 1 ) \, \cpcorb12\sE \Big] \cdot [ K_X + \Delta ]^{d-2} = 0
\end{align*}
holds.
This concludes the proof. \qed

\end{document}